\documentclass[a4paper]{article}
\usepackage[all]{xy}\usepackage[latin1]{inputenc}        
\usepackage[dvips]{graphics,graphicx}
\usepackage{amsfonts,amssymb,amsmath,color,mathrsfs, amstext}
\usepackage{amsbsy, amsopn, amscd, amsxtra, amsthm,authblk}
\usepackage{enumerate,algorithmicx,algorithm}
\usepackage{algpseudocode}
\usepackage{upref}
\usepackage{geometry}
\usepackage[displaymath]{lineno}
\usepackage{float}
\usepackage{yhmath}
\usepackage{booktabs}
\usepackage{exscale}
\usepackage{relsize}
\usepackage{bm}
\usepackage{multirow}
\usepackage{makecell}
\usepackage[colorlinks,
            linkcolor=red,
            anchorcolor=red,
            citecolor=red
            ]{hyperref}

\numberwithin{equation}{section}

\newtheorem{theorem}{Theorem}

\newtheorem{remark}{Remark}

\begin{document}

\title{A kernel-independent sum-of-Gaussians method by de la Vall\'ee-Poussin sums}

\author[1]{Jiuyang Liang}
\author[2]{Zixuan Gao}
\author[1,3]{Zhenli Xu\thanks{xuzl@sjtu.edu.cn}}
\affil[1]{School of Mathematical Sciences, Shanghai Jiao Tong University, Shanghai, 200240, P. R. China}
\affil[2]{Zhiyuan College, Shanghai Jiao Tong University, Shanghai 200240, China}
\affil[3]{Institute of Natural Sciences and MOE-LSC, Shanghai Jiao Tong University, Shanghai, 200240, P. R. China}

\date{}
\maketitle

\begin{abstract}
Approximation of interacting kernels by sum of Gaussians (SOG) is frequently required in many applications
of scientific and engineering computing in order to construct efficient algorithms for kernel summation or convolution problems. In this paper, we propose a kernel-independent SOG method by introducing the de la Vall\'ee-Poussin sum
and Chebyshev polynomials. The SOG works for general interacting kernels and the lower bound of Gaussian bandwidths  is tunable and thus the Gaussians can be easily summed by fast Gaussian algorithms.
The number of Gaussians can be further reduced via the model reduction based on the balanced truncation based on
the square root method. Numerical results on the accuracy and model reduction efficiency show attractive performance
of the proposed method.


\end{abstract}

 \section{Introduction}\label{sec:intro}
For a given smooth  function $f(x)$ for $x\in D$ with $D$ a finite interval and an error tolerance $\varepsilon$,
we consider the approximation of this function by the sum of Gaussians (SOG),
\begin{equation}\label{sog}
\max\limits_{x\in D}\left|f(x)-\sum\limits_j w_j e^{- t_j x^2}\right|<\varepsilon\max\limits_{x\in D}|f(x)|,
\end{equation}
where $w_j$ and $1/\sqrt{t_j}$ are the weight and the bandwidth of the $j$th Gaussian, respectively.
Over the past decades, the SOG approximation has attracted wide interest since it can be useful in many
applications of scientific computing such as convolution integral in physical space \cite{Beylkin2009Fast,Cerioni2012Efficient,Yong2018The}, the kernel summation
problem \cite{H1999A,Yarvin1999AN} and efficient nonreflecting boundary conditions for wave equations \cite{Alpertt2000RAPID,Jiang2004Fast,Jiang2008Efficient,Lubich2002Fast}.
Many kernel functions in these problems have the form of the radial function $f(x)=f(\|\bm{x}\|)$
for $\bm{x}\in \mathbb{R}^d$, such as the power kernel $\|\bm{x}\|^{-s}$ with $s>0$, the Hardy multiquadratic
$\sqrt{\|\bm{x}\|^2+s^2}$, the Mat\'ern kernel \cite{Chen2014Fast} and the thin-plate spline
$\|\bm{x}\|^2\log\|\bm{x}\|$. An SOG approximation to these kernels is particularly useful because
a Gaussian kernel can simply achieve the separation of variables such that a convolution of $f(x)$ with it can be computed as the summation of products of $d$ one-dimensional integrals, dramatically reducing the cost.

The solution of \eqref{sog} has been an extensively studied subject in literature \cite{Beylkin2005Approximation,Beylkin2010Approximation,Braess1995Asymptotics,Braess2009On,Dietrich2005Approximation,Evans1980On,articleFernandez,articleGonchar,David1979Least,Varah1985On,Wiscombe1977Exponential}. One way to construct the SOG expansion is via the best rational approximation or the sum-of-poles
approximation, which is based on the fact that the rational approximation for the Laplace transform
of function $f(x)$ has explicit expression, and the inverse transform gives the sum-of-exponentials
approximation. When the kernel is radially symmetric, the SOG approximation is
equivalent to the SOE approximation by a simple change of variable $y=\sqrt{x}$.
If one does not restrict the bound of the Gaussian bandwidths, the integral representation of the kernel
can be a significant tool. For example, the power function $x^{-s}$ has the following inverse Laplace
transform expression \cite{Beylkin2010Approximation},
\begin{equation}\label{powerfunction}
x^{-s}=\dfrac{1}{\Gamma(s)}\int_{-\infty}^{\infty}e^{-e^{t}x+s t}dt,
\end{equation}
with $\Gamma(\cdot)$ being the Gamma function. A suitable quadrature rule to Eq.\eqref{powerfunction}
yields an explicit discretization to obtain a sum of exponentials. Constructing the SOG approximation
from integral representation is superficially attractive due to the controllable high accuracy with the number of Gaussians.
The integral form for general kernels by the inverse Laplace transform can be found in Dietrich and Hackbusch \cite{Braess2009On}.

Another important approach for the SOG is the least squares approximation. A straightforward use
of the least squares method will be less accurate due to the difficulty of solving the ill-conditioned matrix.
The least squares problem can be solved by employing the divided-difference factorization and the modified
Gram-Schmidt method to significantly improve the accuracy \cite{Wiscombe1977Exponential}.
Greengard {\it et al.} \cite{Yong2018The} developed a black-box approach for the SOG of radially symmetric kernels. This approach allocates a set
of logarithmically equally spaced points $t_j$ lying on the positive real axis, then a set of sampling points $x_i$
is constructed via adaptive bisections. Then the fitting matrix $A$ with entry $A_{ij}=e^{-t_j x_i}$
and the right hand vector $b$ with $b_i=f(x_i)$ are constructed. After one obtains the weights by solving
the least squares problem, the square root method in model reduction \cite{GLOVER1984All} is
introduced to reduce the number of exponentials and achieve a near-optimal SOE approximation.

Function approximation using Gaussians is a highly nonlinear problem and the design of
such approximations for $t_j$ such that it is bounded by a positive number (i.e., the lower bound of  bandwidths is independent of the number of Gaussians) is nontrivial.
In applications with fast Gauss transform (FGT) \cite{Greengard1991The,ArticleRoussos}
to calculate the kernel summation problem with kernel approximated by the SOG, a small lower bound of
bandwidths will seriously reduce the performance of the algorithm. Due to this issue,
it was reported \cite{boyd2010uselessness} that the FGT is less useful for summing Gaussian radial basis function
series, in spite that Gaussian radial basis function interpolation
has been widely employed in many applications such as machine learning problems.
In this work, we propose a novel kernel-independent SOG method which preserves both high precision and
a tunable lower bound of bandwidths. We construct the Gaussian approximation of the kernel by
the de la Vall\'ee-Poussin (VP) sum \cite{de1919leccons,Natanson1965Constructive} via a variable substitution.
The variable substitution introduces a parameter $n_c$ which allows to tune the minimal bandwidth of
the Gaussians. Additionally, the model reduction can be further used to reduce the number of Gaussians under a specified accuracy level, achieving an optimized SOG approximation.
	
The remainder of the paper is organized as follows. In Section 2, we derive the kernel-independent
SOG approach and discuss the technique details. In Section 3, we perform numerical examples
which demonstrate the performance of the new SOG scheme for different kernels.
Concluding remarks are given in Section 4.

\section{Sum-of-Gaussians approximation}\label{SOGbyVPSum}
In this section, we consider the problem of approximating functions on a finite real interval by linear combination
of Gaussians, i.e., we approximate kernel function $f(x)$ by sum of $p$ Gaussians,
\begin{equation}\label{kernel_approx}
f_p(x)=\sum_{j=1}^p w_j e^{-x^2/s_j^2},
\end{equation}
where $w_j$ and $s_j$ are weights and bandwidths, respectively. We define $s_{p}=\min\limits_{j}|s_{j}|$ as the minimal bandwidth of the Gaussians.

\subsection{de la Vall\'ee-Poussin sums}
We introduce the VP sum to obtain an SOG expansion. Let $f(x)$ be a smooth function defined
on the positive axis, $x\geq0$. We assume $f(x)$ has a limit at infinity.
Without loss of generality, we assume the limit is zero, $\lim_{x\rightarrow\infty}f(x)=0$.
We introduce the variable substitution
\begin{equation}\label{variable-convert}
x=\sqrt{-n_c\log\left(\dfrac{1+\cos t}{2}\right)}, ~t=[0,\pi],
\end{equation}
where $t\leftrightarrow x$ is an one-to-one mapping. Parameter $n_c$ is a positive constant, which determines
the specified lower bound of bandwidths.
Let $\varphi(t)=f(x)$, which is smooth on $[0,\pi]$. We have $\varphi(0)=f(0)$ and $\varphi(\pi)=f(\infty)=0$.
We make an even prolongation of $\varphi(t)$ on $[-\pi,\pi]$ such that $\varphi(t)$ can be treated as an even periodic
function with $2\pi$ period in $(-\infty,\infty)$.

The VP sum of $\varphi(t)$ is defined by (see \cite{Al2001Approximation} for reference),
\begin{equation}\label{Vallee}
V_n[\varphi(t)]=\dfrac{1}{n}\sum\limits_{\ell=n}^{2n-1}S_\ell[\varphi(t)],
\end{equation}
where
\begin{equation}
S_\ell[\varphi(t)]=\sum_{k=0}^{\ell}a_k\cos(k t),
\end{equation}
is the Fourier partial sum of $\varphi(t)$ and the Fourier coefficients $a_k$ are defined by,
\begin{equation}\label{a_alpha}
a_k=\begin{cases}
\dfrac{1}{\pi}\mathlarger{\int}_0^\pi \varphi(t)dt, ~~\text{for} ~k=0,\\\\
\dfrac{2}{\pi}\mathlarger{\int}_0^{\pi}\varphi(t)\cos(k t)dt, ~~\text{for} ~k\geq 1.
\end{cases}
\end{equation}
The VP sum \eqref{Vallee} can be reorganized into two components,
\begin{equation}\label{VS_TwoP}
V_n[\varphi(t)]=S_n[\varphi(t)]+\sum\limits_{\ell=1}^{n-1}\left(1-\dfrac{\ell}{n}\right)a_{n+\ell}\cos\left[(n+\ell)t\right].
\end{equation}
We substitute the inverse mapping $t=\arccos (2e^{-x^2/n_c}-1)$ to Eq. \eqref{VS_TwoP} and find,
\begin{equation}\label{Kn}
f_p(x)=\sum\limits_{\ell=0}^{n}a_\ell T_\ell\left(2e^{-x^2/n_c}-1\right)+\sum\limits_{\ell=1}^{n-1}\left(1-\dfrac{\ell}{n}\right)a_{n+\ell}T_{n+\ell}\left(2e^{-x^2/n_c}-1\right)
\end{equation}
where $f_p(x)=V_n[\varphi(t)]$ is an approximation of $f(x)$, and by Theorem \ref{vpc}
it is an SOG with $p=2n$. Here, $T_m(x)$ is the Chebyshev polynomial of order $m$ defined by,
\begin{equation}\label{Tn}
T_m(x)=\cos(m\arccos(x))=\sum\limits_{\ell=0}^{\lfloor m/2 \rfloor}(-1)^\ell\binom{m-\ell}{2\ell}x^{m-2\ell}(1-x^2)^\ell.
\end{equation}

By substituting Eq. \eqref{Tn} into Eq. \eqref{Kn} and rearranging the coefficients, we obtain the following Theorem \ref{vpc}.

\begin{theorem}\label{vpc}
Function $f_p(x)$ defined in Eq. \eqref{Kn} can be written as a sum of Gaussians,
	\begin{equation}\label{VPVPVP}
	f_p(x)=\sum_{j=0}^{2n-1}w_je^{-j x^2/n_c},
\end{equation}
with $p=2n$. Here, coefficient $w_j$ is given by,
	\begin{equation}\label{Expansion-coefficients}
	w_j=\begin{cases}	a_0+\sum\limits_{\ell=1}^n(-1)^\ell a_\ell+\sum\limits_{\ell=1}^{n-1}(-1)^{n+\ell}\left(1-\dfrac{\ell}{n}\right)a_{n+\ell},~~~~\text{\emph{for}}~j=0,\\\\
	2^{2j} \sum\limits_{\ell=j}^{n}(-1)^{\ell-j}\dfrac{\ell }{\ell+j}\mathlarger{\binom{\ell+j}{\ell-j}}a_\ell
	+\sum\limits_{\ell=1}^{n-1}c_n^{j\ell} a_{n+\ell},
	~~~~\text{\emph{for}}~1\leq j\leq n,\\\\	
	\sum\limits_{\ell=j-n}^{n-1}c_n^{j\ell} a_{n+l},~~~~\text{\emph{for}}~j>n,
	\end{cases}
	\end{equation}
with
	$$c_n^{j\ell}=(-1)^{n+\ell-j}\left(1-\dfrac{\ell}{n}\right)\dfrac{(n+\ell)}{n+\ell+j} \mathlarger{\binom{n+\ell+j}{n+\ell-j}}2^{2j}.$$
\end{theorem}

Eq. \eqref{VPVPVP} gives the expression of the SOG approximation with the bandwidth of $j$th Gaussian being
$s_j=\sqrt{n_c/j}$ for $j>0$.  Note that the only approximation
introduced in the SOG method is the numerical calculation of the
Fourier coefficients Eq. \eqref{a_alpha}, and the fast cosine transform can be employed for
rapid evaluation of these coefficients.

\begin{remark}
The minimal bandwidth in the Gaussians is $s_p=\sqrt{n_c/(2n-1)}$, thus $n_c$ determines the
lower bound of all bandwidths. It asymptotically becomes constant if one sets $n_c\varpropto n$.
\end{remark}

\subsection{Error estimate}
We discuss the error estimate of the VP sum $V_n[\varphi(t)]$ to approximate $\varphi(t)$.
We assume that $\varphi(t)$ is twice continuously differentiable in $[-\pi,\pi]$ except at $t=0$,
and analyze the errors when the function is not differentiable and first-order differentiable at $t=0$,
respectively.

When $\varphi(t)$ is not differentiable at $t=0$, it was shown that $V_n[\varphi(t)]$ still converges to $\varphi(t)$ uniformly
on $\mathbb{R}$, but the rate of convergence at $t=0$ is much slower than that at any other point. In this case, the error estimate of the VP sum was given in Boyer and Goh \cite{boyer2011generalized},
and it is not difficult to follow the estimate to obtain the following result,
\begin{equation} \label{case1}
V_n[\varphi(0)]-\varphi(0)=-\dfrac{\ln 2}{n\pi}\sqrt{n_c}f'(0)+O\left(n^{-\frac{3}{2}}\right).
\end{equation}

We now consider the case of $\varphi(t)$ being first-order differentiable at $t=0$, and there is
$f'(0)=0$ since $f(x)$ is even. Many radial basis functions satisfy the condition $f'(0)=0$ such as the inverse multiquadratic kernel
and the Mat\'ern kernel with $\nu\geq1$. In this case, the leading term in \eqref{case1} vanishes, and the
error order is higher. Actually, Theorem \ref{estimate} shows that the error is the second order of
convergence with respect to $1/n$, instead of $O(n^{-3/2})$.

\begin{theorem}\label{estimate}
	Suppose that $V_n[\varphi(t)]$ is the $n$th VP sum of $\varphi(t)$ which is twice-differentiable on $[0,\pi]$
	with period $2\pi$ and defined through Eq. \eqref{variable-convert} by $f(x)$. If $f'(0)=0$, then we have
	\begin{eqnarray} \label{order}
	&&V_n[\varphi(0)]-\varphi(0)=O\left(n^{-2}\right), ~~\hbox{and}, \\
	&&V_n[\varphi(t)]-\varphi(t)=o\left(n^{-2}\right), ~~\hbox{for}~t\neq 0.
	\end{eqnarray}
\end{theorem}

\begin{proof}
	We have $S_n[\varphi(t)]$ and $V_n[\varphi(t)]$ denoting the $n$th Fourier partial sum and the VP sum
	of $\varphi(t)$, respectively. Introduce the Fej\'er partial sum
	\begin{equation}
	\sigma_n[\varphi(t)]=\dfrac{1}{n+1}\sum\limits_{\ell=0}^{n}S_\ell[\varphi(t)].
	\end{equation}
	The error of the VP sum can be written as
	\begin{equation}\label{eq1}
	V_n[\varphi(t)]-\varphi(t)=2\sigma_{2n}[\varphi(t)]-\sigma_{n}[\varphi(t)]-\varphi(t).
	\end{equation}
	By using the Fej\'er kernel representation \cite{Natanson1965Constructive}, one has,
	\begin{equation}\label{eq2}
	 \sigma_n[\varphi(t)]-\varphi(t)=\dfrac{1}{n\pi}\int_{-\pi}^{\pi}\big[\varphi(t+\xi)-\varphi(t)\big]\dfrac{\sin^2\frac{n\xi}{2}}{2\sin^2\frac{\xi}{2}}d\xi.
	\end{equation}
	Substituting Eq.\eqref{eq2} into Eq.\eqref{eq1}, one gets,
	\begin{equation}\label{eq3}
	V_n[\varphi(t)]-\varphi(t)=\dfrac{1}{n\pi}\int_{-\pi}^{\pi}\big[\varphi(t+\xi)-\varphi(t)\big]\dfrac{\cos n\xi-\cos 2n\xi}{4\sin^2\frac{\xi}{2}}d\xi.
	\end{equation}
	
	Consider the case of $t=0$. Eq.\eqref{eq3} can be decomposed into two parts $I_1$ and $I_2$ with
	\begin{equation}
	I_1=\dfrac{1}{n\pi}\int_0^\pi(\varphi(\xi)-\varphi(0))\dfrac{\cos n\xi-\cos 2n\xi}{4\sin^2\frac{\xi}{2}}d\xi
	\end{equation}
	and
	\begin{equation}
	I_2=\dfrac{1}{n\pi}\int_{-\pi}^0(\varphi(\xi)-\varphi(0))\dfrac{\cos n\xi-\cos 2n\xi}{4\sin^2\frac{\xi}{2}}d\xi.
	\end{equation}
	
First, we focus on $I_1$ and write it as  $I_1=I_{11}+I_{12}$ such that,
	\begin{equation} \label{i11}
	I_{11}=\dfrac{1}{n\pi}\int_0^\pi(\varphi(\xi)-\varphi(0))(\cos n\xi-\cos 2n\xi)\left(\dfrac{1}{4\sin^2\frac{\xi}{2}}-\dfrac{1}{\xi^2}\right)d\xi
	\end{equation}
	and
	\begin{equation}
	I_{12}=\dfrac{1}{n\pi}\int_0^\pi(\varphi(\xi)-\varphi(0))(\cos n\xi-\cos 2n\xi)\dfrac{1}{\xi^2}d\xi.
	\end{equation}
In Eq. \eqref{i11}, $1/4\sin^2(\xi/2)-1/\xi^2$ remains positive between $0.08$ and $0.15$ for $\xi\in[0,\pi]$, and each part of the integrand in $I_{11}$ is continuous on closed interval $[0,\pi]$. By the first mean value theorem for integrals, there exists a positive number $M_1$ such that
	\begin{equation}
	I_{11}=\dfrac{M_1}{n\pi}\int_0^\pi(\varphi(\xi)-\varphi(0))(\cos n\xi-\cos 2n\xi)d\xi.
	\end{equation}
By integration by parts, one gets
	\begin{equation}\label{eq4}
	\begin{split}
	I_{11}=&\dfrac{M_1}{n\pi}\left[\left(\dfrac{1}{n}\sin n\xi-\dfrac{1}{2n}\sin 2n\xi\right)\dfrac{\varphi(\xi)-\varphi(0)}{\xi}\Bigg|_{\xi=0}^{\pi}\right]-\\
	&\int_0^\pi\dfrac{d\frac{\varphi(\xi)-\varphi(0)}{\xi}}{d\xi}\left(\dfrac{1}{n}\sin n\xi-\dfrac{1}{2n}\sin 2n\xi\right)d\xi\\
	=&\dfrac{M_1}{n^2\pi}\int_0^\pi\dfrac{d\frac{\varphi(\xi)-\varphi(0)}{\xi}}{d\xi}\left(\sin n\xi-\dfrac{1}{2}\sin 2n\xi\right)d\xi\\
	=&O\left(\dfrac{1}{n^2}\right),
	\end{split}
	\end{equation}
	where the last two steps employ the conditions that $\varphi'(0)=0$ and $\varphi''(0)$ exists.
	
For $I_{12}$, one can decompose $I_{12}$ into two parts $I_{121}$ and $I_{122}$ such that,
	\begin{equation}
	I_{121}=\dfrac{1}{n\pi}\int^{\frac{1}{n}}_0(\varphi(\xi)-\varphi(0))(\cos n\xi-\cos 2n\xi)\dfrac{1}{\xi^2}d\xi,
	\end{equation}
	and
	\begin{equation}
	I_{122}=\dfrac{1}{n\pi}\int_{\frac{1}{n}}^\pi(\varphi(\xi)-\varphi(0))(\cos n\xi-\cos 2n\xi)\dfrac{1}{\xi^2}d\xi.
	\end{equation}
Note that $\cos n\xi-\cos 2n\xi$ is a monotonically increasing function which is non-negative on $[0, 1/n]$. Due to the existence of $f''(t)$, the rest part of the integrand of $I_{121}$ is integrable and bounded. One employs the second mean value theorem for integrals to $I_{121}$, and finds that there exists a positive $M_2\leq 1/n$ such that
	\begin{equation}
	I_{121}=\dfrac{1}{n\pi}(\cos 1 -\cos 2)\int^{\frac{1}{n}}_{M_2}\dfrac{\varphi(\xi)-\varphi(0)}{\xi^2}d\xi=O\left(\dfrac{1}{n^2}\right).
	\end{equation}
Note that $[\varphi(\xi)-\varphi(0)]/\xi^2$ is bounded on $[1/n,\pi]$. There exists a positive number $M_3$ such that
	
	\begin{equation}
	|I_{122}|\leq\dfrac{M_3}{n\pi}\Bigg|\int_{\frac{1}{n}}^{\pi}(\cos n\xi-\cos 2n\xi)d\xi\Bigg|=\dfrac{M_3(2\sin 1-\sin 2)}{2n^2\pi}=O\left(\dfrac{1}{n^2}\right).
	\end{equation}
	
Finally, combining these estimates, one gets,
	\begin{equation}
	I_1=I_{11}+I_{12}=I_{11}+I_{121}+I_{122}=O\left(\dfrac{1}{n^2}\right).
	\end{equation}
Similarly, it holds,
	\begin{equation}
	I_2=O\left(\dfrac{1}{n^2}\right).
	\end{equation}
Substituting these estimates for $I_1$ and $I_2$ into Eq. \eqref{eq3} completes the proof of the first equality in Eq. \eqref{order}. For the case $t\neq 0$, the proof is similar and we omit the details.
	
\end{proof}

\begin{remark}
If the smooth function $f(x)$ has a limit at infinity, our method works on the whole positive-axis, thus the interval $D$ could be an arbitrary subset of $\mathbb{R}^+$. If the limit doesn't exist, the above approach of constructing the SOG expansion has low accuracy because of the discontinuity of the transformed function $\varphi(t)$. In order to achieve higher accuracy, one could truncate $f(x)$ at a required point, then connect a fast decreasing function behind the cutoff point such that a new function $f^*(x)$ to localize the kernel function. The higher-order differentiable properties of $f^*(x)$ can be kept.
\end{remark}

\subsection{Model reduction method}
In this section, we consider to reduce the number of Gaussians in the SOG expansion.
We apply the square root method in model reduction \cite{GLOVER1984All,Moore1981Principal}
for the purpose, which can achieve a near optimal approximation. That is, we find
a $q$-term Gaussians for Eq.\eqref{VPVPVP} such that
	\begin{equation}\label{mr_1}
	\sum_{j=1}^{2n-1}w_je^{-j x^2/n_c}\approx \sum_{\ell=1}^q \widetilde{w}_\ell e^{-x^2/s_{\ell}^2},
	\end{equation}
where $q<2n-1$ and the minimal bandwidth $s_{q}=\min\limits_{\ell}|s_\ell|\approx \sqrt{n_c/(2n-1)}$.
Note that the constant $j=0$ term is ignored here.

Let $y=x^2$. The model reduction procedure first writes the left hand side of
Eq. \eqref{mr_1} (excluding $j=0$) into the SOE expansion with parameter $y$, and then
introduces the Laplace transform on it to obtain a sum-of-poles representation,
\begin{equation}\label{s-o-p}
		\mathscr{L} \left[\sum_{j=1}^{2n-1}w_je^{-jy/n_c}\right] =\sum\limits_{j=1}^{2n-1}\dfrac{w_j}{z+j/n_c}.
\end{equation}
The sum-of-poles representation in Eq.\eqref{s-o-p} can be simply expressed as
the following transfer function in the linear dynamical system problem,
\begin{equation}
		\bm{c}(z\bm{I}-\bm{A})^{-1}\bm{b}=\sum\limits_{j=1}^{2n-1}\dfrac{w_j}{z+j/n_c},
\end{equation}
where $\bm{A}$ is a diagonal matrix, $\bm{b}$ and $\bm{c}$ are column and row vectors, respectively and they reads,
\begin{equation}\label{coef}
		\begin{split}		
&\bm{A} =-\text{diag}\left\{\frac{1}{n_c},\frac{2}{n_c},\cdots,\frac{(2n-1)}{n_c}\right\},\\
&\bm{b} =\left(\sqrt{|w_1|},~\sqrt{|w_2|},~\cdots,~\sqrt{|w_{2n-1}|}\right)^{T},\\
&\bm{c} =\left(\text{sign}(w_1) \sqrt{|w_1|},~\text{sign}(w_2) \sqrt{|w_2|},~\cdots,
~\text{sign}(w_{2n-1}) \sqrt{|w_{2n-1}|}\right).
		\end{split}
\end{equation}
		
With the transfer function, the associated dynamical system with coefficients given by Eq.\eqref{coef}
can be reduced by the balanced truncation \cite{Moore1981Principal} where the square root method
can be used for the purpose. This technique can be simply generalized to reduce the number of poles
and thus the Gaussians in the SOG. The first step of the algorithm is solving
two Lyapunov equations for two matrices $\bm{P}$ and $\bm{Q}$,
\begin{equation}
	\bm{AP}+\bm{PA}^{*}+\bm{b}\bm{b}^{*}=0,~~\bm{A}^{*}\bm{Q}+\bm{QA}+\bm{c}^{*}\bm{c}=0,
\end{equation}
where $*$ indicates the conjugate transpose. The second step is to find a balancing transformation matrix $\bm{X}$ by employing the square root method such that the singular value of product $\bm{PQ}$ 
can be calculated. These procedures lead to a reduced system with coefficients $\widetilde{\bm{A}}^{q\times q}$, $\widetilde{\bm{b}}^{q\times 1}$
and $\widetilde{\bm{c}}^{1\times q}$, which are defined as the $q\times q$, $q\times 1$, $1\times q$ leading blocks of $\bm{XAX}^{-1}$, $\bm{Xb}$, and $\bm{cX}^{-1}$, respectively. The corresponding transfer function
$\widetilde{\bm{c}}(z\widetilde{\bm{I}}-\widetilde{\bm{A}})^{-1}\widetilde{\bm{b}}$ satisfies
		\begin{equation}		 \sup\limits_{z=i\mathbb{R}}\left|\widetilde{\bm{c}}(z\widetilde{\bm{I}}-\widetilde{\bm{A}})^{-1}\widetilde{\bm{b}}-\bm{c}(z\bm{I}-\bm{A})^{-1}\bm{b}\right|\leq \delta,
		\end{equation}
for a given tolerance $\delta$. By employing the eigendecomposition and the inverse Laplace transform,
one accomplishes the model reduction procedure, resulting in an optimized SOG approximation with $q$
Gaussians Eq. \eqref{mr_1}.

We remark that the model reduction technique was originally designed for sum-of-poles approximation
and the optimality of the resulting SOE approximation in the $L^\infty$ norm is guaranteed
by well-known results in control theory \cite{GLOVER1984All,Xu2013A}. However, since all the nodes lie
in the left half of the complex plane, we are allowed to apply it directly to the reduction of SOG approximation
because of the aforementioned connection between these two types of approximations \cite{Yong2018The}.
The implementation of model reduction can also be obtained by employing the balance of Moore or the orthogonal-diagonal approach \cite{Moore1981Principal,GUGERCIN2009SIAM}.

\section{Numerical results}
In this section, we present numerical results to illustrate the performance of the SOG approximation method developed in this paper. Due to the requirement of high-precision matrix manipulation, we employ the Multiple Precision Toolbox \cite{MP} in order to implement the model reduction procedure. Surprisingly, We find the parameters of Gaussians after the model reduction can be well represented by double-precision floating point numbers. The computer code is released as open source, which is available at the link https://github.com/ZXGao97. All the calculations are performed on a Intel TM core of clock rate $2.50$ GHz with $24$ GB of memory.

Four different kernels are used to measure the performance of the algorithm.
These are the Gaussian kernel $f_\mathrm{gau}$ which has a small bandwidth, the inverse multiquadratic kernel $f_\mathrm{imq}$,
the Ewald splitting kernel $f_\mathrm{ewd}$ and the Mat\'ern kernel $f_\mathrm{mat}$, expressed as follows,
\begin{eqnarray}
&\displaystyle f_\mathrm{gau}(x)=e^{-x^2/h^2}, \\
&\displaystyle f_\mathrm{imq}(x)=\frac{1}{\sqrt{1/2+x^2}}, \\
&\displaystyle f_\mathrm{ewd}(x)=\frac{\text{erf}(\alpha x)}{x}, \\
&\displaystyle f_\mathrm{mat}(x)=\dfrac{(\sqrt{2\nu}|x|)^{\nu}K_\nu(\sqrt{2\nu}|x|)}{2^{\nu-1}\Gamma(\nu)},
\end{eqnarray}
where $\text{erf}(x)=(2/\sqrt{\pi})\int_0^x \exp(-u^2)du$ is the error function, $K_\nu$ is the modified Bessel function of
the second kind of order $\nu$ and $\Gamma$ is the Gamma function. We take parameters $h=0.1$, $\alpha=1$ and $\nu=2$ in the
calculations. We remark that the Gaussian and the inverse multiquadratic kernels are mostly used radial basis functions
which have been used in a broad range of data science and engineering problems \cite{Hu1998The,Scholkopf2002Comparing}.
The Ewald splitting kernel $f_\mathrm{ewd}$ is the long-range part of the well-known Ewald summation \cite{Darden1993Particle,EwaldDie,JinLiXuZhao2020} for
Coulomb interaction and the parameter $\alpha$ describes the inverse of cutoff radius.
Lastly, the Mat\'ern kernel is also a radial basis function and often used as a covariance function in
modeling Gaussian process \cite{Chen2014Fast}, where the parameter $\nu$ describes the smoothness of the kernel.

We begin with the performance of the SOG to approximate the exact kernels with
the increase of $p$. To assess the accuracy, we compute the maximal
relative error $\epsilon_\infty$ of the resulted SOG approximation $f_p(x)$ with $p=2n$
in the case of without the model reduction, which is defined by,
\begin{equation}\label{MaximalError}
\epsilon_\infty =\dfrac{\max \big\{|f_p(x_i)-f(x_i)|,i=1,\cdots,M\big\} }{\max\big\{|f(x_i)|, i=1,\cdots,M\big\}},
\end{equation}
where $\{x_i, i=1,\cdots, M\}$ are monitoring points randomly distributed from $[0,1]$ and
we take $M=1000$. The error can be viewed as a scaling approximation of the continuous $L^\infty$ norm.

The results are given in Figure \ref{TheFirstFigure}. In panel (a), we
show the maximal relative errors for the four kernels with the increase of $p$.
The parameter $n_c$ is set to be
$n_c=\lceil n/4\rceil$, and thus the minimal bandwidth is fixed to be $s_p\approx \sqrt{1/8}$, asymptotically
independent of $n$.
We observe that high accuracy of the SOG approximation is achieved for all four kernels and
that the convergence rate is very promising. It is mentioned that the Gaussian kernel
$f_\mathrm{gau}$ has very small bandwidth $h=0.1$, but the SOG with all $s_j\gg h$ has
rapid convergence when $p$ is bigger than 100. In panel (b), the maximal relative error
is shown as a function of the minimal bandwidth, where the total number of Gaussians is fixed to
be $p=10000$ and the value $n_c$ is varying to tune the minimal bandwidth $s_p$. For
the small Gaussian and the Mat\'ern kernels, the observed accuracy
is significantly improved when the minimal bandwidth is reduced. Whereas, the SOG approximation for
the other two kernels seems not sensitive to the varying of minimal bandwidth since they are smoother functions
and the high-frequency components in Fourier space are very small.

\begin{figure}[!htbp]
	\centering
		\includegraphics[width=0.49\linewidth]{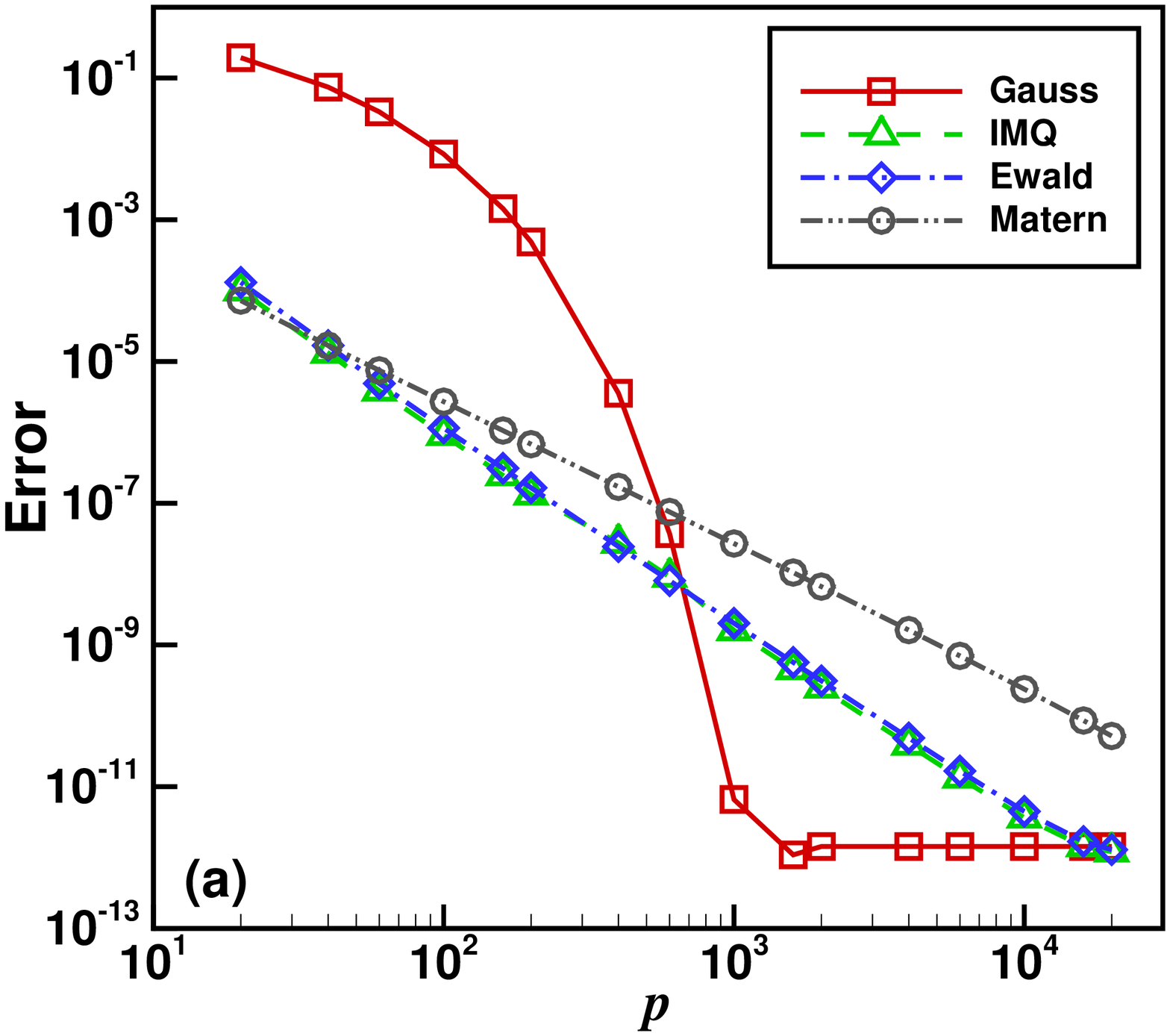}
        \includegraphics[width=0.49\linewidth]{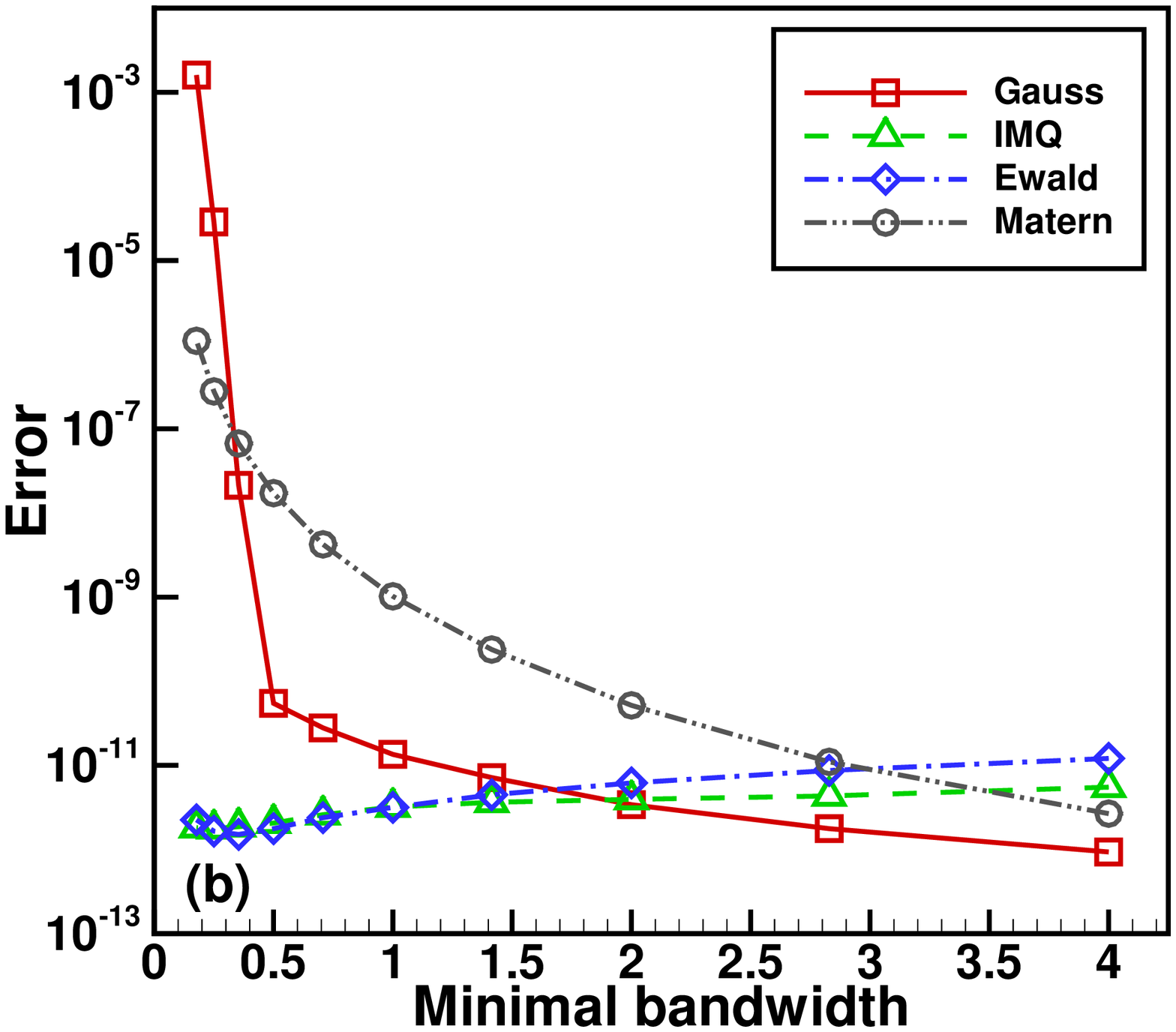}
	\caption{ Dependence of the maximal relative error $\epsilon_\infty$ of the SOG approximation as: (a)
function of $p$ with fixed minimal bandwidth; and (b) function of the minimal bandwidth with fixed number
of Gaussians. Data is shown for four kernels: the small Gaussian (Gauss),
the inverse multiquadric (IMQ), the Ewald splitting (Ewald) and the Mat\'ern kernel (Matern).	}
	\label{TheFirstFigure}
\end{figure}

In practical calculations, the magnitude of the coefficients of Gaussians has close relation to the round-off error.
It is necessary to check the effect of varying minimal bandwidth on the coefficient magnitudes.
Figure \ref{TheSecondFigure} shows the maximal absolute value of the coefficients (maximum weight),
$w_\mathrm{max}=\max \{ |w_j|, j=0,1,\cdots,p-1 \}$, of the SOG expansion as function of minimal bandwidth $s_p$.
The two panels shows results of  fixed numbers of Gaussians, $p=20$ and $40$, respectively.
The relation between $w_\mathrm{max}$ and the minimal bandwidth is substantially different for these kernels.
Generally, as can be observed in Figure \ref{TheSecondFigure} (ab), the maximal coefficient magnitude
increases with the number of Gaussians used for the SOG approximation. Within the calculated range of $s_p$,
the maximum weight varies in about 3 number of digits for both $p=20$ and 40 cases.

\begin{figure}[!htbp]
	\centering
		\includegraphics[width=0.49\linewidth]{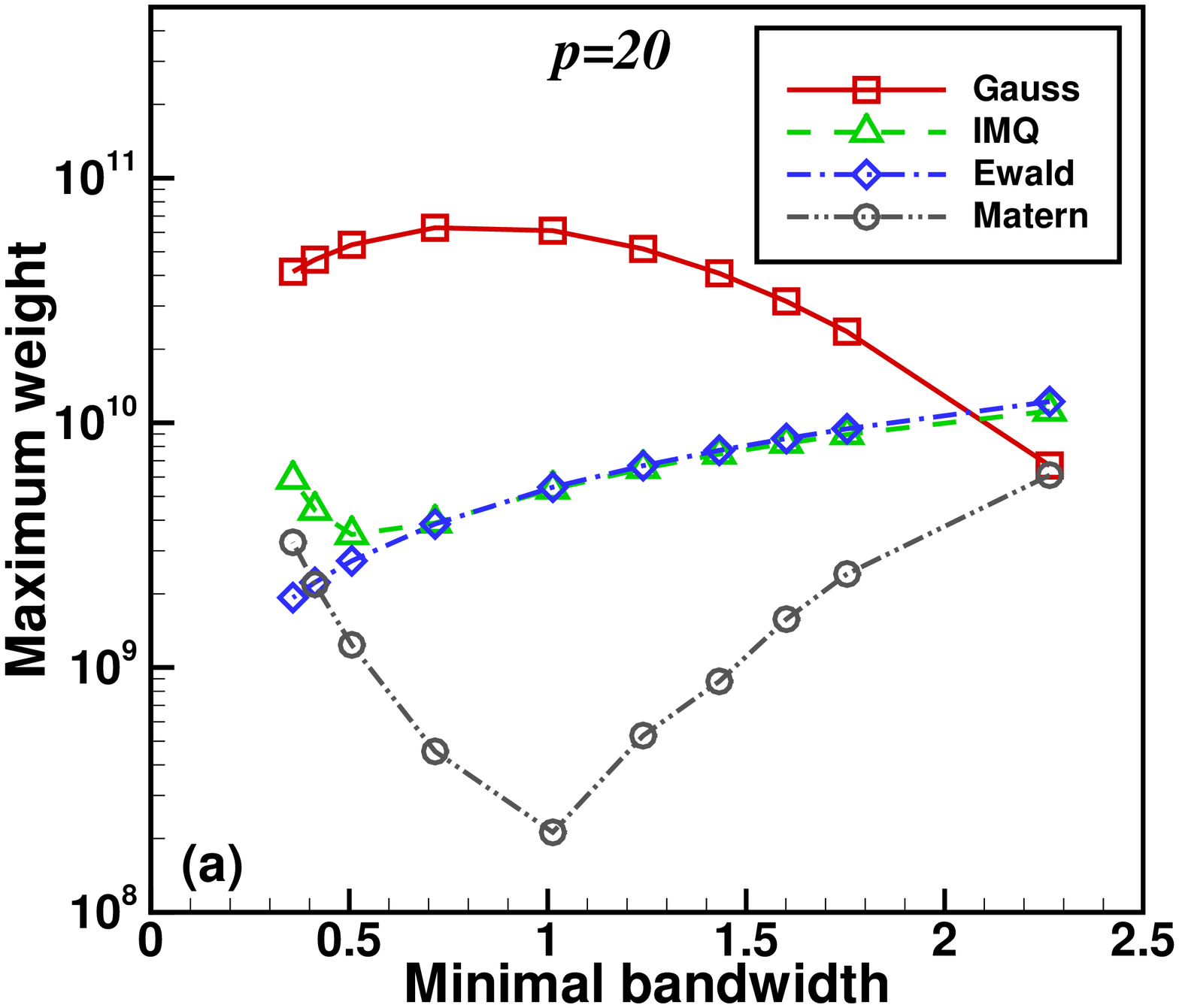}
		\includegraphics[width=0.49\linewidth]{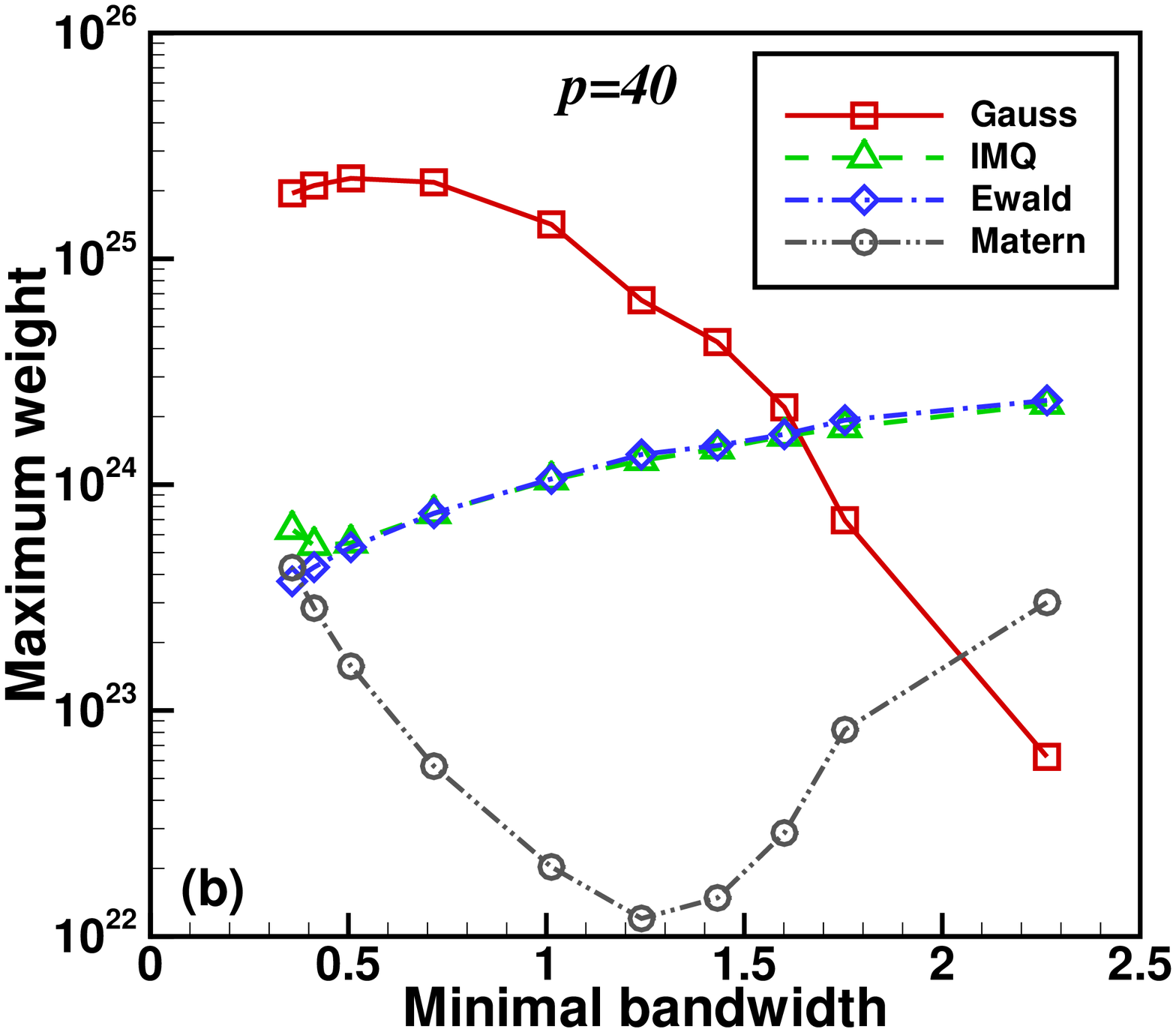}
	\caption{The maximal weight of the SOG approximation $w_\mathrm{max}$ as a function of minimal bandwidth
with two numbers of Gaussians: (a) $p=20$ and (b) $p=40$. Data for four kernels are calculated: small Gaussian,
inverse multiquadratic, Ewald splitting and Mat\'ern.
	}
	\label{TheSecondFigure}
\end{figure}

We next compare the SOG approximation constructed by the VP sum with the SOG by the least squares method (LSM).
For the LSM, we employ the complete orthogonal decomposition to compute the low-rank approximation of the fitting matrix
such that the ill-conditioned matrix can be well resolved. In Figure \ref{TheThirdFigure}(a-d), we show the comparison results
for the four different kernels, where the relative errors as function of $x$ and $x\in[0,1]$ are displayed.
Both the VP sum and the LSM are shown with $p=200$ and $800$.
Clearly, the VP-based SOG approximations provide more accurate results for all the four different kernels, thanks
to the analytical manner of the VP sum.

\begin{figure}[!htbp]
	\centering
		\includegraphics[width=0.49\linewidth]{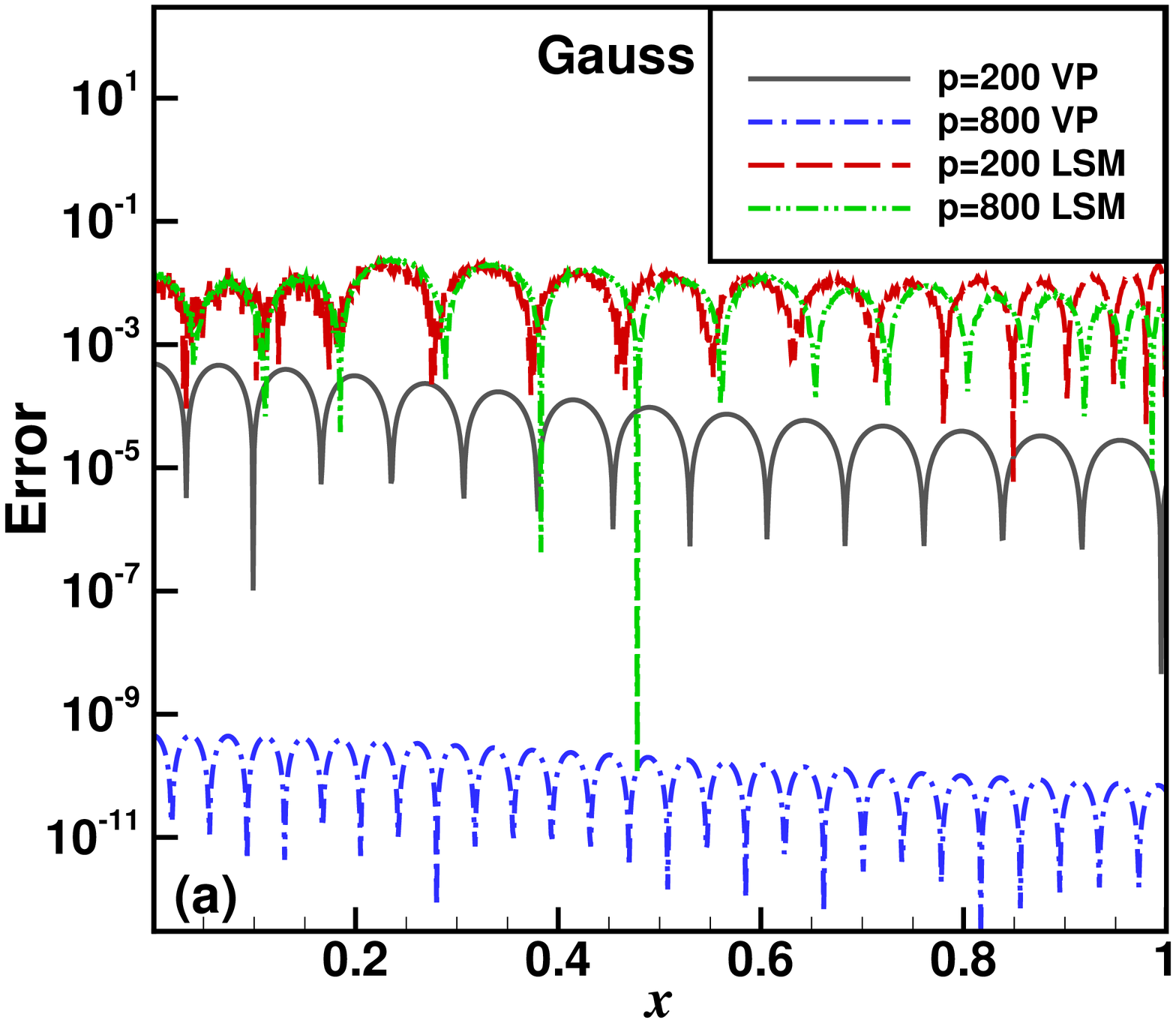}
		\includegraphics[width=0.49\linewidth]{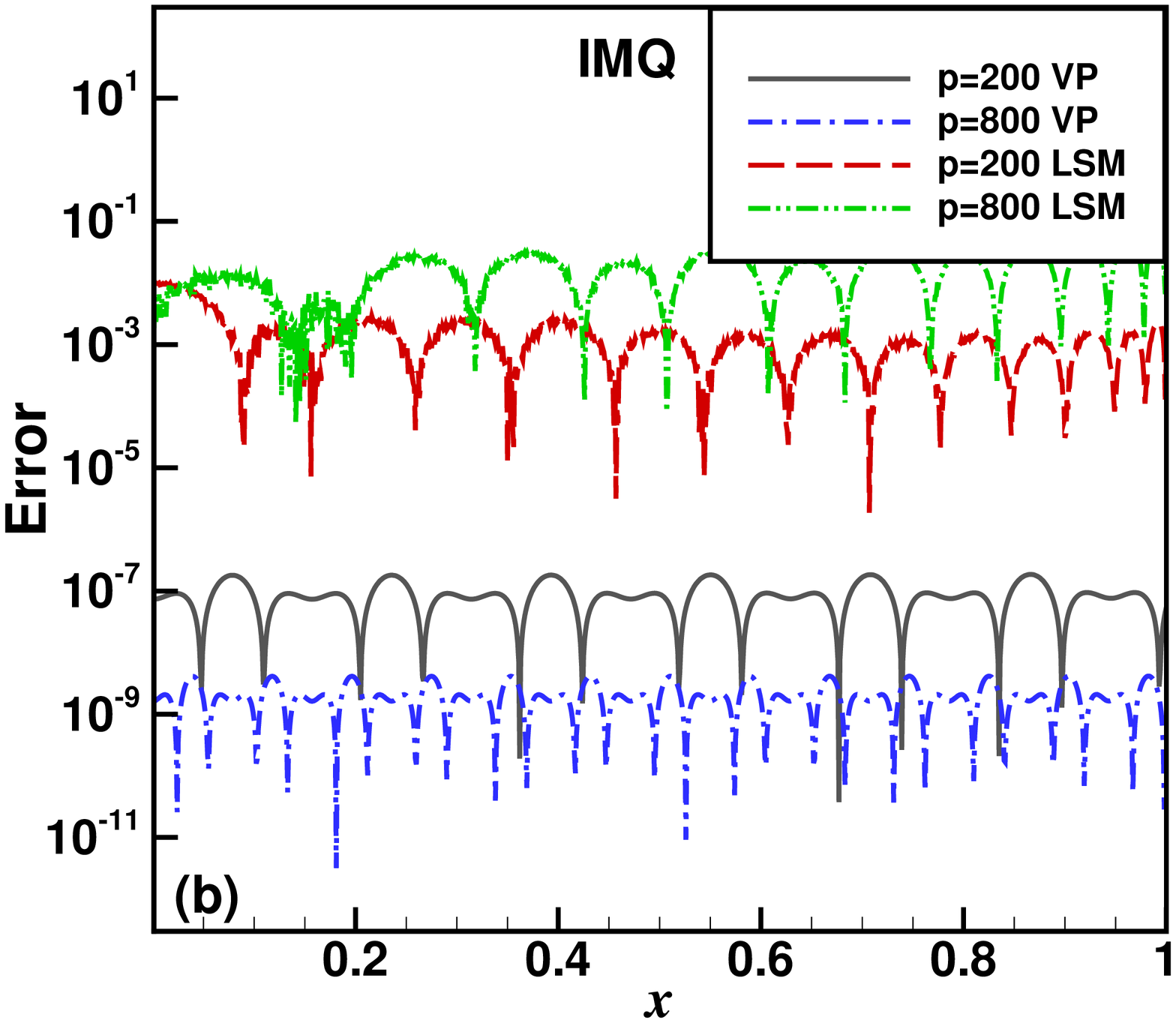}
		\includegraphics[width=0.49\linewidth]{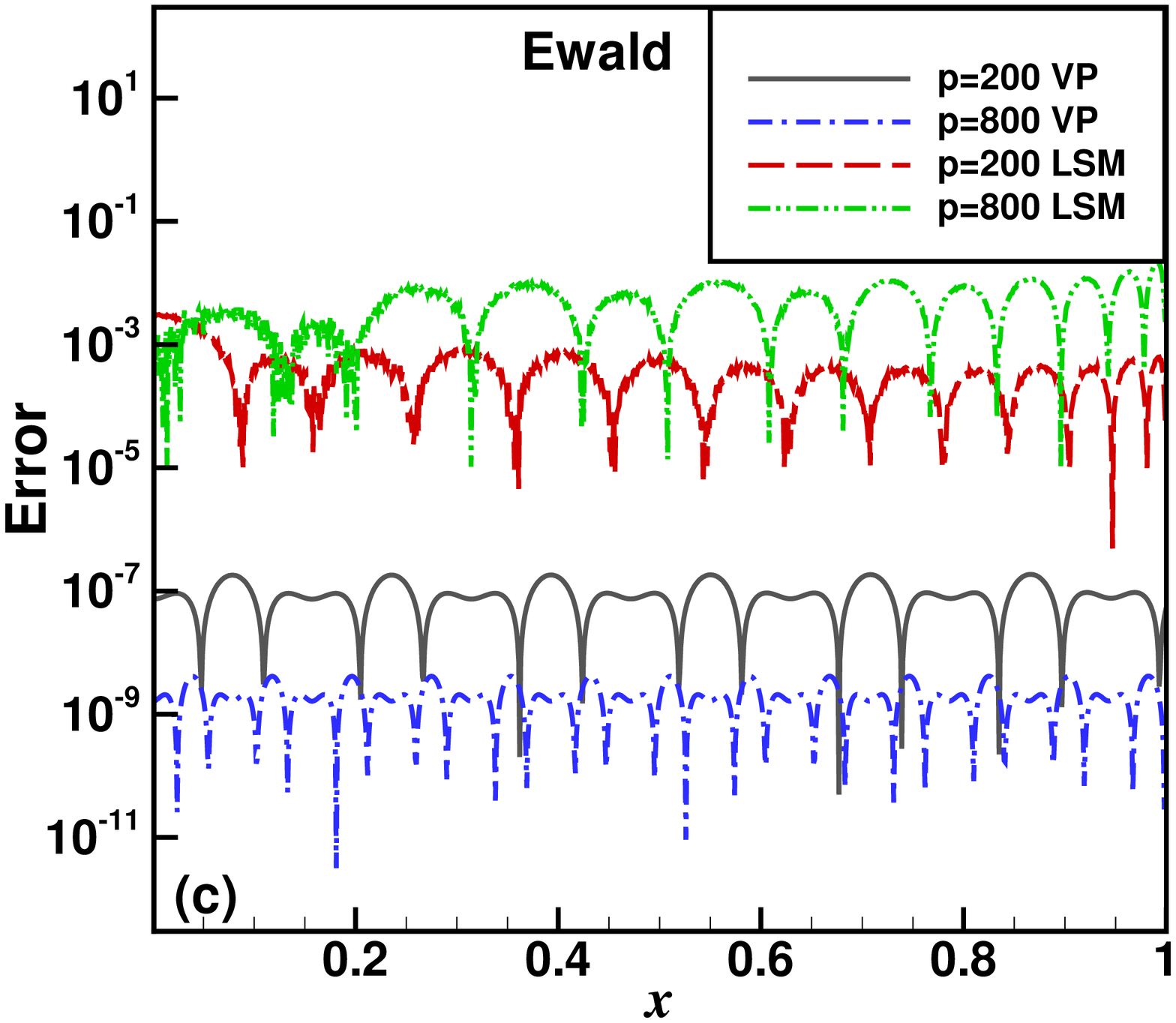}
		\includegraphics[width=0.49\linewidth]{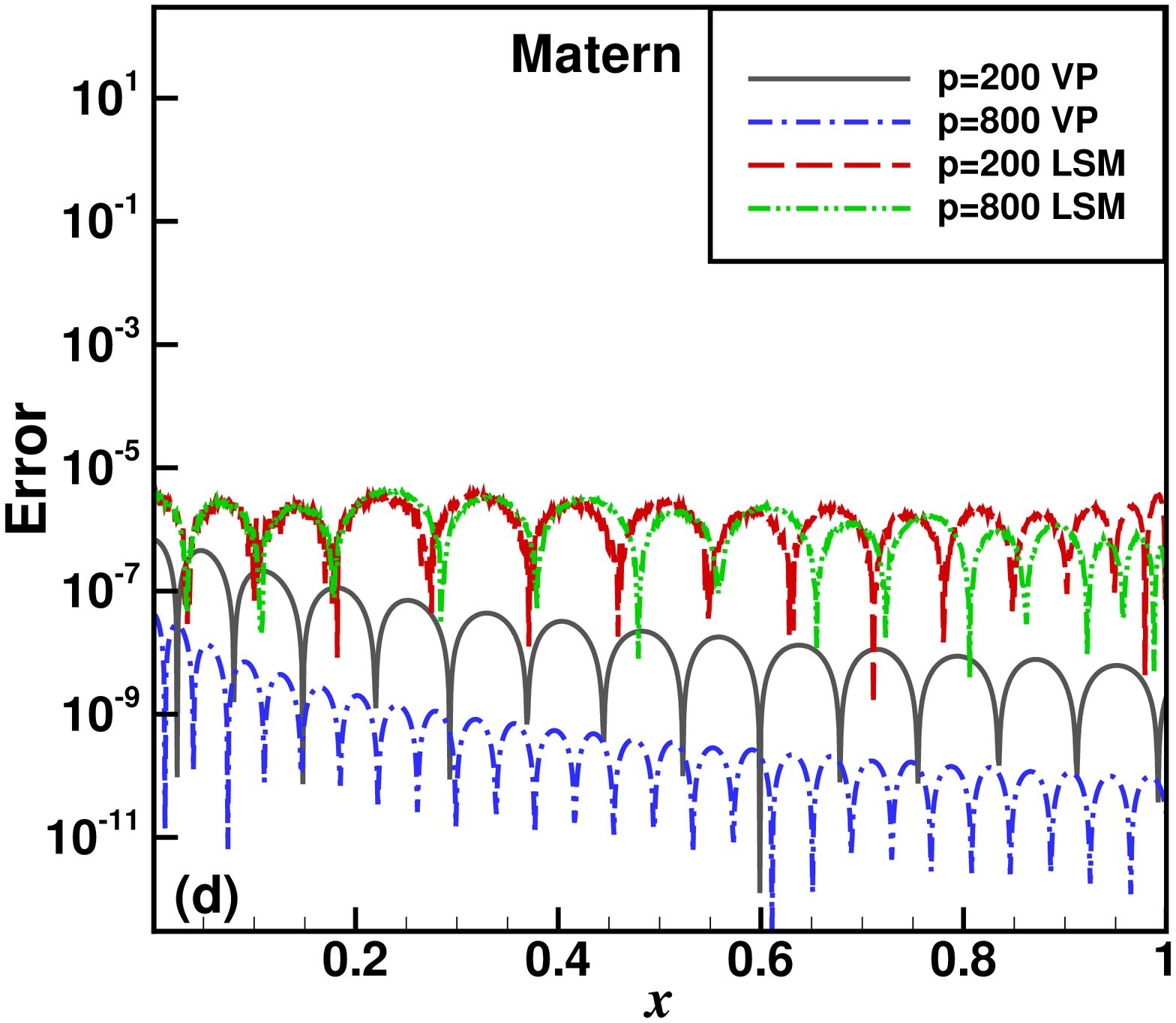}
	\caption{Relative error as function of $x$ for the SOG approximations via VP-sums and LSM for four different kernels:
(a) Gaussian; (b) inverse multiquadric; (c) Ewald splitting; and (d) Mat\'ern. Results of the SOG methods
with 200 and 800 Gaussians are shown.
	}
	\label{TheThirdFigure}
\end{figure}

Finally, we investigate the efficiency of the model reduction method used for the
SOG approximation based on the VP sum. In Tables \ref{tab:InvQuad} and \ref{tab:Matern}, we present
the results with the model reduction of the SOG approximation for the inverse multiquadric kernel and
Mat\'ern kernel where $100$ initial Gaussians (with $n=50$ for the VP sum) are used.
The data of maximum weight $\widetilde{w}_{\max}$, the minimal bandwidth $s_q$, and the maximum relative
error $\epsilon_\infty$ are shown for different reduced numbers of Gaussians $q$.
More than $30\%$ of Gaussians for inverse multiquadratic and $50\%$ of Gaussians for Mat\'ern could be reduced
under the same level of $\varepsilon_\infty$ as the original SOG. Interestingly,
the maximum weight $\widetilde{w}_{\max}$ dramatically becomes small number with the use of the model
reduction, and can be well represented by double-precision numbers. With the decrease of Gaussians,
$\widetilde{w}_{\max}$ will become smaller, and the minimum bandwidth $s_q$ increases
for the inverse multiquadratic kernel and slightly decreases for the Mat\'ern kernel.
These results demonstrate the model reduction method is efficient to obtain the optimized
SOG approximation.

\begin{table}[ht!]
	\centering
	\caption{Model reduction  with $100$ initial Gaussians for $f_\mathrm{imq}$ }
	\label{tab:InvQuad}
	\begin{tabular}{c|ccc}
		\hline\hline
	 Reduced number $q$ & \hspace{1.5cm}$\widetilde{w}_{\max}$\hspace{1cm} & \hspace{1cm}$s_q$\hspace{1cm} & \hspace{1cm} $\epsilon_\infty$ \hspace{1cm} \\ \hline
		$100$ & 5.96e+68 & 0.361 & 2.36e-6  \\
		$90$ & 37.5 & 0.201 & 2.36e-6 \\
		$70$ & 13.7 & 0.346 & 2.66e-6 \\
		$50$ & 6.90 & 0.363 & 2.34e-5 \\
		$30$ & 2.31 & 0.421 & 1.87e-4 \\
		$10$ & 2.31 & 0.665 & 1.03e-2 \\ \hline\hline
	\end{tabular}
\end{table}

\begin{table}[ht!]
	\centering
	\caption{Model reduction with $100$ initial Gaussians for $f_\mathrm{mat}$}
	\label{tab:Matern}
	\begin{tabular}{c|ccc}
		\hline\hline
		Reduced number $q$ & \hspace{1.5cm}$\widetilde{w}_{\max}$\hspace{1cm} & \hspace{1cm}$s_q$\hspace{1cm} & \hspace{1cm} $\epsilon_\infty$ \hspace{1cm}   \\  \hline
  	  $100$ & 5.70e+64 & 0.361 & 3.87e-6  \\
		$90$ & 0.335 & 0.131 & 3.87e-6 \\
		$70$ & 0.467 & 0.122 & 3.88e-6 \\
		$50$ & 0.309 & 0.113 & 3.89e-6 \\
		$30$ & 0.246 & 0.116 & 5.68e-6 \\
		$10$ & 0.274 & 0.153 & 1.84e-5 \\ \hline\hline
	\end{tabular}
\end{table}

\section{Conclusions}\label{conclusions}
We have developed a high-accurate and kernel-independent SOG method for which the minimal bandwidth of Gaussians is controllable. This method is constructed by using the variable substitution and VP sums. The number of Gaussians is further reduced by employing the model reduction via square root approach. Such approximations can be combined with the Hermite expansion \cite{inbook,Hille2002} and the fast algorithms of \cite{Greengard1991The,SpivakThe} to achieve efficient, accurate and robust methods for the fast evaluation of kernel summation and convolution problems, which will be studied in our future work.

\section*{Acknowledgement}
The authors acknowledge the financial support from the Natural Science Foundation of China (Grant No. 12071288), the Strategic Priority Research Program of CAS (Grant No. XDA25010403), Shanghai Science and Technology Commission (Grant No. 20JC1414100) and the support from the HPC Center of Shanghai Jiao Tong University. The authors thank Prof. Shidong Jiang for some helpful comments.


\end{document}